\documentclass[reqno]{amsart}
\usepackage{amssymb,amsmath,amsthm,amscd,latexsym,amsfonts,multicol}
\usepackage{mathtools}
\usepackage[english]{babel}
\usepackage[T1]{fontenc}
\usepackage{inputenc}
\usepackage[arrow,matrix]{xy}
\usepackage{amsaddr}
\usepackage{graphicx}
\usepackage{cite}

\newtheorem{thm}{Theorem}[section]
\newtheorem{cor}[thm]{Corollary}

\newtheorem{prop}[thm]{Proposition}
\theoremstyle{definition}
\newtheorem{defn}[thm]{Definition}

\theoremstyle{remark}

\numberwithin{equation}{section}
\DeclareMathOperator{\Der}{Der}

\DeclareMathOperator{\Ann}{Ann}

\begin{document}

\title[Solvable Leibniz algebras with quasi-filiform of maximum length nilradical]{Solvable Leibniz algebras whose nilradical is a quasi-filiform Leibniz algebra of maximum length}

\author{Q.K. Abdurasulov$^1$, J.Q. Adashev$^2$, J.M. Casas$^3$, B.A. Omirov$^4$, }
\address{$^{1,2}$ Institute of Mathematics, Tashkent, 100125, Uzbekistan, \\
abdurasulov0505@mail.ru, adashevjq@mail.ru} 
\address{$^{3}$ Department of Applied Mathematics I, E. E. Forestal, University of Vigo, \\ 36005 Pontevedra, Spain, jmcasas@uvigo.es}
\address{$^{4}$ National University of Uzbekistan, Tashkent, 100174, Uzbekistan, omirovb@mail.ru}

\begin{abstract}  We describe solvable Leibniz algebras whose nilradical is a quasi-filiform Leibniz algebra of maximum length.
\end{abstract}



\maketitle





\section{Introduction}

Leibniz algebras are generalizations of Lie algebras and they have been introduced by J.-L. Loday in \cite{Loday} as a non-antisymmetric version of Lie algebras. These algebras preserve a unique property of Lie algebras - the right multiplication operators are derivations. Many classical results of the theory of Lie algebras were extended to the case of Leibniz algebras. For instance, the analogue of Levi's theorem for Leibniz algebras was proved by Barnes \cite{Bar}. He showed that any finite-dimensional Leibniz algebra is decomposed into the semidirect sum of solvable radical and semisimple Lie subalgebra. Therefore, the biggest challenge in the classification problem of finite-dimensional Leibniz algebras is the study of solvable part.
Due to \cite{Nulfilrad}, where the method of the description of solvable Lie algebras with a given nilradical developed in \cite{Mub} was extended to the case of Leibniz algebras, the problem of classification of solvable Leibniz algebras reduces to study of nilpotent one.

The inherent property of non Lie Leibniz algebras is the existence of the non-trivial ideal, generated by the squares of elements of an algebra.

The approach to description of Lie algebras with a given nilradical is based on the method developed by Malcev in \cite{Mal} and Mubarakzyanov in \cite{Mub}. To descriptions of solvable Lie algebras with various types of nilradicals were devoted papers \cite{NdWi},\cite{Rubin},\cite{Shab2, SnWi, SnWi2, TrWi, WaLiDe}. For the case of solvable Leibniz algebras we have  similar results, that is, there are classifications of solvable Leibniz algebras with various type of nilradical,
like Heisenberg nilradical \cite{Dunbar}, filiform \cite{Camacho6}, naturally graded filiform \cite{filrad},\cite{Ladra}, triangular \cite{Iqbol}, direct sum of null-filiform algebras \cite{Abror2}, some classes of naturally graded quasi-filiform Leibniz algebras \cite{Shab} etc.

The starting point of the present paper is quasi-filiform Leibniz algebras of maximal length \cite{Cabezas2}, \cite{Camacho} considered further as nilradicals of solvable Leibniz algebras.

Since solvable Leibniz algebras whose nilradical is quasi-filiform Lie algebra of maximal length were classified in \cite{Khosiyat}, we consider the case when nilradical is quasi-filiform non Lie algebra of maximum length.

Throughout the paper vector spaces and algebras are finite-dimensional over the field of the complex numbers. Moreover, in the table of multiplication of any algebra the omitted products are assumed to be zero
and, if it is not noted, we consider non-nilpotent solvable algebras.

\section{Preliminaries}

In this section we briefly give necessary definitions and preliminary results which can be found also in \cite{Camacho}, \cite{Nulfilrad}, \cite{Loday}, \cite{Shab}.

\begin{defn} A  vector space with a bilinear bracket $(L,[-,-])$ over a field $\mathbb{F}$ is called a \emph{Leibniz algebra} if for any $x,y,z\in L$ the Leibniz identity
\[ \big[x,[y,z]\big]=\big[[x,y],z\big] - \big[[x,z],y\big] \]  holds.
\end{defn}

Further we   use the notation
\[ {\mathcal L}(x, y, z)=[x,[y,z]] - [[x,y],z] + [[x,z],y].\]
It is obvious that Leibniz algebras are determined by the identity ${\mathcal L}(x, y, z)=0$.

From the Leibniz identity we conclude that the elements $[x,x], [x,y]+[y,x]$ for any $x, y \in L$ belong to the right annihilator (denoted by $\Ann_r(L)$) of an algebra $L$. Moreover, it is easy to see that $\Ann_r(L)$ is a two-sided ideal of $L$.

The notion of a derivation in the case of Leibniz algebras is defined as usual, that is, a linear map $d \colon L \rightarrow L$ of a Leibniz algebra $L$ is said to be a \emph{derivation} if it satisfies  \begin{equation}\label{eq0}
d([x,y])=[d(x),y] + [x, d(y)] \ \mbox{for any} \ x, y, \in L.
\end{equation}

Note that the right multiplication operator $\mathcal{R}_x \colon L \to L, \mathcal{R}_x(y)=[y,x],  y \in L$, is a derivation.

For a given Leibniz algebra $L$ we consider the lower central and the derived series
\[L^1=L, \ L^{k+1}=[L^k,L],  \ k \geq 1, \qquad \qquad
L^{[1]}=L, \ L^{[s+1]}=[L^{[s]},L^{[s]}], \ s \geq 1,
\]
respectively.

\begin{defn} A Leibniz algebra $L$ is said to be
\emph{nilpotent} (respectively, \emph{solvable}), if there exists $n\in\mathbb N$ ($m\in\mathbb N$) such that $L^{n}=0$ (respectively, $L^{[m]}=0$).
\end{defn}

Evidently, for an $n$-dimensional nilpotent Leibniz algebra $L$ we have $L^{n+1}=0$.

The maximal nilpotent ideal of a Leibniz algebra is called the \emph{nilradical} of the  algebra.

Let $R$ be a solvable Leibniz algebra with nilradical $N$. We denote by $Q$ the complementary vector space  of the nilradical $N$ to the algebra $R$.
 Let us consider the restrictions to $N$ of the right multiplication operator on an element $x \in Q$ (denoted by $\mathcal{R}_{{x |}_{N}}$). Thanks to \cite{Nulfilrad} we know that for any $x \in Q$, the operator $\mathcal{R}_{{x |}_{N}}$ is a non-nilpotent  derivation of $N$.

Let $\{x_1, \dots, x_m\}$ be a basis of $Q$, then for any scalars $\{\alpha_1, \dots, \alpha_m\}\in
\mathbb{C}\setminus\{0\}$, the matrix $\alpha_1\mathcal{R}_{{x_1 |}_{N}}+\dots+\alpha_m\mathcal{R}_{{x_m|}_{N}}$ is non-nilpotent, which means that the elements $\{x_1, \dots, x_m\}$ are \emph{nil-independent} \cite{Mub}. Therefore, the dimension of $Q$  is bounded by the maximal number of nil-independent derivations of the nilradical $N$ (see \cite[Theorem 3.2]{Nulfilrad}). Moreover, similar to the case of Lie algebras, for a solvable Leibniz algebra $R$ the inequality $\dim N \geq \frac{1}{2}\dim R$ holds.

%

Below we define the notion of a quasi-filiform Leibniz algebra.

\begin{defn} A Leibniz algebra $L$ is called quasi-filiform if $L^{n-2}\neq 0$ and $L^{n-1}=0$, where $n=\dim L.$
\end{defn}

%

A Leibniz algebra $L$ is called $\mathbb{Z}$-graded if $L =\oplus_{i\in \mathbb{Z}}V_i,$ where $[V_i, V_j]\subseteq V_{i+j}$ for any $i, j \in \mathbb{Z}$ with a finite number of non-null spaces $V_i.$

A gradation $L = V_{k_1}\oplus\cdots\oplus V_{k_t}$ of a Leibniz algebra $L$ is called \emph{connected gradation} if $V_{k_i}\neq0$ for any $i \  (1\leq i\leq t)$ and the number $l(\oplus L):=l(V_{k_1}\oplus\cdots\oplus V_{k_t})=k_t-k_1 + 1$ is called \emph{the length of the gradation}.

\begin{defn} A Leibniz algebra $L$ is called to be \emph{of maximum length} if $max\{l(\oplus L)\ \mbox{such that}\  L = V_{k_1}\oplus\cdots\oplus V_{k_t}\  \mbox{is a connected gradation}\} = \dim(L).$
\end{defn}

In the following theorem we give the classification of quasi-filiform non Lie Leibniz algebras of maximum length given in \cite{Cabezas2} and \cite{Camacho}.

\begin{thm} \label{thmquasi-fil} An arbitrary $n$-dimensional quasi-filiform non Lie Leibniz algebra of maximum length is isomorphic to one algebra of the following pairwise non-isomorphic algebras of the families:
\[M^{1,\delta}:\left\{\begin{array}{ll}
[e_1, e_1] =e_{n}, & [e_{n-1}, e_1] =e_{2},\\[1mm]
[e_i,e_1]=e_{i+1}, & 2\leq i\leq n-3,\\[1mm]
[e_{n-1}, e_{n-1}] =\delta e_{4}, & \delta\in\{0,1\},\\[1mm]
[e_i,e_{n-1}]=\delta e_{i+3}, & 2\leq i\leq n-5,\\[1mm]
 \end{array}\right. \qquad M^{2,\lambda}:\left\{\begin{array}{ll}
[e_i,e_1]=e_{i+1}, & 1\leq i\leq n-3,\\[1mm]
[e_{n-1},e_1]=e_{n}, &\\[1mm]
[e_1,e_{n-1}]=\lambda e_{n}, & \lambda\in \mathbb{C},\\[1mm]
 \end{array}\right.\]
\[M^{3,\alpha}:\left\{\begin{array}{ll}
[e_1, e_1] =e_{2}, &\\[1mm]
[e_i,e_1]=e_{i+1}, & 3\leq i\leq n-1,\\[1mm]
[e_1,e_i]=-e_{i+1}, & 3\leq i\leq n-1,\\[1mm]
[e_3, e_{3}] =\alpha e_{6}, & \alpha=0, \ if\  n>6,\\[1mm]
& \alpha\in\{0,1\}, \ if\  n=6,\\[1mm]
 \end{array}\right. \qquad
M^4:\left\{\begin{array}{ll}
 [e_i,e_1]=e_{i+1}, & 1\leq i\leq n-3,\\[1mm]
  [e_{1}, e_{n-1}] =e_{n},&\\[1mm]
  \end{array}\right.\]
where $\{e_1,e_2,\dots,e_{n}\}$ is a basis of the algebra.
\end{thm}

For solvable Leibniz algebras with $n$-dimensional nilradical $M$ and with $s$-dimensional complemented space to the nilradical we use the notation $R(M,s)$.

Thanks to work \cite{Abdurasulov} we already have the classification of solvable Leibniz algebras $R(M^4,1)$, while there is no solvable Leibniz algebra of the family $R(M^4,2)$. Namely, we have the following theorem.

\begin{thm} \label{thmM^4} An arbitrary solvable Leibniz algebra $R(M^4,1)$ is isomorphic to one of the following pairwise non-isomorphic algebras of the family:

$$R(M^4,1)(a_2,\dots,a_{n-1}):\left\{\begin{array}{ll}
[e_i,x]=\sum\limits_{j=i+1}^{n-2}a_{j-i+1}e_j, & 1\leq i\leq n-2,\\[1mm]
[e_{n-1},x]=e_{n-1},& [e_{n},x]=e_{n},\\[1mm]
[x,e_{n-1}]=-e_{n-1},& [x,x]=a_{n-1}e_{n-2},\\[1mm]
  \end{array}\right.$$
where the first non-vanishing parameter $a_2, \dots, a_{n-1}$ in the family $R(M^4,1)(a_2,\dots,a_{n-1})$ can be scaled to $1$.
\end{thm}

Moreover, in \cite{Shab} the description of solvable Leibniz algebras with nilradical $M^{3, 0}$ is presented.

\begin{thm} \label{thm31}An arbitrary solvable Leibniz algebra $R(M^{3,0}, 1)$ is isomorphic to one of the following pairwise non-isomorphic algebras:

$$R_1(M^{3,0}, 1):\left\{\begin{array}{llll}
[e_1,x]=e_1, & [e_2,x]=2e_2, & [e_i,x]=(i-n)e_i,& 3\leq i\leq n-1,\\[1mm]
[x,x]=e_{n},& [x,e_1]=-e_1,& [x,e_i]=(n-i)e_i,& 3\leq i\leq n-1,\\[1mm]
 \end{array}\right.$$
$$R_2(M^{3,0}, 1)(\alpha_i,\beta_j):\left\{\begin{array}{llll}
[e_{1},x]=\alpha_1e_{2},& [e_i,x]=e_i+\alpha_2 e_{i+2}+\sum\limits_{k=i+3}^{n}\beta_{k-i-2}e_k,& 3\leq i\leq n,\\[1mm]
[x,e_{1}]=\alpha_3e_{2},& [x,e_i]=-e_i-\alpha_2 e_{i+2}-\sum\limits_{k=i+3}^{n}\beta_{k-i-2}e_k,& 3\leq i\leq n,\\[1mm]
[x,x]=\alpha_4e_{2}, & \alpha_2\in \{0,\pm1\},&\\[1mm]
\end{array}\right.$$
$$R_3(M^{3,0}, 1)(\alpha):\left\{\begin{array}{llll}
[e_1,x]=e_1,& [e_2,x]=2e_2,& [x,e_1]=-e_1,\\[1mm]
[e_i,x]=(i-3+\alpha)e_i,& 3\leq i\leq n, \\[1mm]
[x,e_i]=(3-i-\alpha)e_i,& 3\leq i\leq n,  & \alpha\in \mathbb{C},\\[1mm]
 \end{array}\right.$$
$$R_4(M^{3,0}, 1):\left\{\begin{array}{lll}
[e_1,x]=e_1+e_3,&\\[1mm]
[e_{2},x]=2e_{2},&[e_i,x]=(i-2)e_i, &3\leq i\leq n,\\[1mm]
[x,e_1]=-e_1-e_3,&  [x,e_i]=(2-i)e_i,& 3\leq i\leq n.\\[1mm]
 \end{array}\right.$$
\end{thm}

\begin{thm} \label{thm32}An arbitrary algebra of the family $R(M^{3,0},2)$ is isomorphic to the following algebra:
$$R(M^{3,0}, 2):\left\{\begin{array}{lll}
[e_1,x_1]=e_1,&&\\[1mm]
[e_2,x_1]=2e_2,& [e_i,x_1]=(i-3)e_i,& 3\leq i\leq n,\\[1mm]
[x_1,e_1]=-e_1,& [x_1,e_i]=(3-i)e_i,& 3\leq i\leq n,\\[1mm]
[e_1,x_2]=e_1,&&\\[1mm]
 [e_2,x_2]=2e_2,& [e_i,x_2]=(i-2)e_i,& 3\leq i\leq n,\\[1mm]
[x_2,e_1]=-e_1,& [x_2,e_i]=(2-i)e_i,& 3\leq i\leq n,\\[1mm]
 \end{array}\right.$$
\end{thm}

In order to simplify our further calculations for the algebras $M^{1,\delta}$ and $M^{3,1}$, by taking the change of basis in the following form, respectively:
\[\begin{array}{lllll}
e_1^\prime=e_1, & e_2^\prime=e_{n-1}, & e_{i}^\prime=e_{i-1},  & 3\leq i\leq n-1,&e_{n}^\prime=e_{n},\\[1mm]
e_1^\prime=e_1, & e_{6}^\prime=e_{2}, & e_{i}^\prime=e_{i+1},  & 2\leq i\leq 5,&\\[1mm]
 \end{array}\]
we obtain the table of multiplication of the algebras $M^{1,\delta}$ and $M^{3,\alpha}$, which we   use throughout the paper:
\[M^{1,\delta}:\left\{\begin{array}{ll}
[e_1, e_1] =e_{n}, & \\[1mm]
[e_i,e_1]=e_{i+1}, & 2\leq i\leq n-2,\\[1mm]
[e_i,e_{2}]=\delta e_{i+3}, & 2\leq i\leq n-4,\ \delta\in\{0,1\},\\[1mm]
 \end{array}\right.\quad M^{3,1}:\left\{\begin{array}{ll}
[e_1, e_1] =e_{6}, &\\[1mm]
[e_i,e_1]=e_{i+1}, & 2\leq i\leq 4,\\[1mm]
[e_1,e_i]=-e_{i+1}, & 2\leq i\leq 4,\\[1mm]
[e_2, e_{2}] =e_{5}. & \\[1mm]
 \end{array}\right.\]

\section{Solvable Leibniz algebras whose nilradical is quasi-filiform non Lie Leibniz algebra of maximum length.}

This section is devoted to the classification of solvable Leibniz algebras whose nilradical is non Lie Leibniz algebra of maximum length. Due to Theorems \ref{thmM^4}, \ref{thm31} and \ref{thm32} we only need to consider solvable Leibniz algebras with nilradicals $M^{1,\delta}, M^{2,\lambda}$ and $M^{3,1}$.

\subsection{Derivations of algebras $M^{i,*}, i=1, 2, 3$}

\

In order to begin the description of solvable Leibniz algebras with nilradicals $M^{i,*}, i=1, 2, 3$ we need to know their derivations.

\begin{prop} \label{prop1} An arbitrary $d_1\in \Der(M^{1,\delta})$ (respectively, $d_2\in \Der(M^{2,\lambda})$) has the following matrix form:
$$d_1=\left(\begin{array}{ccccccccc}
a_1&0&0&0&\dots &0&a_{n-2}&a_{n-1}&a_{n}\\
0&b_2&b_3&b_4&\dots&b_{n-3}&b_{n-2}&b_{n-1}&b_{n}\\
0&0&a_1+b_2&b_3&\dots&b_{n-4}&b_{n-3}&b_{n-2}&0\\
0&0&0&2a_1+b_2&\dots&b_{n-5}&b_{n-4}&b_{n-3}&0\\
\vdots&\vdots& \ddots&\vdots &\vdots&\vdots&\vdots \\
0&0&0&0&\dots&0&0&(n-3)a_1+b_2&0\\
0&0&0&0&\dots&0&0&a_{n-2}&2a_1\\
\end{array}\right),$$
where if $\delta=1,$ then $b_2=3a_1;$
$$d_2=\left(\begin{array}{cccccccc}
a_1&a_2&\dots &a_{n-3}& a_{n-2}&a_{n-1}&a_{n}\\
0&2a_1&\dots&a_{n-4}&a_{n-3}&0&(\lambda+1)a_{n-1}\\
0&0&\dots&a_{n-5}&a_{n-4}&0&0\\
\vdots&\vdots& \ddots&\vdots &\vdots&\vdots&\vdots \\
0&0&\dots&0&(n-2)a_1&0&0\\
0&0&\dots&b_{n-3}&b_{n-2}&b_{n-1}&b_n\\
0&0&\dots&0&b_{n-3}&0&a_1+b_{n-1}\\
\end{array}\right),$$
where if $\lambda\neq0,$ then $b_{n-3}=0.$
\end{prop}
\begin{proof} It is easy to see that $\{e_1,e_2\}$ are the generator basis elements of the algebra $M^{1,\delta}$.

We put
\[ d_1(e_1)=\sum\limits_{t=1}^{n}a_te_t, \qquad d_1(e_{2})=\sum\limits_{t=1}^{n}b_te_t.\]

From the derivation property \eqref{eq0} we have
\[
d_1(e_n)=d_1([e_1,e_1])=[d_1(e_1),e_1]+[e_1,d_1(e_1)]=2a_1e_n+\sum\limits_{t=3}^{n-1}a_{t-1}e_{t}.
\]

Consider
\[ 0=d_1(e_n,e_1])=[d_1(e_n),e_1]+[e_n,d_1(e_1)]=\sum\limits_{t=4}^{n-1}a_{t-2}e_{t}.\]
Consequently,
\[ a_{t}=0,\quad  2 \leq t \leq n-3.\]

From the derivation property \eqref{eq0} we have
\[d_1(e_3)=d_1([e_2,e_1])=[d_1(e_2),e_1]+[e_2,d_1(e_1)]=(a_1+b_2)e_3+\sum\limits_{t=4}^{n-1}b_{t-1}e_{t}+b_1e_n.\]

Buy induction and the property of derivation \eqref{eq0} we derive
\[d_1(e_i)=((i-2)a_1+b_2)e_i+ \sum\limits_{t=i+1}^{n-1}b_{t-i+2}e_t, \quad 4 \leq i \leq n-1.\]

Consider
$d_1(\delta e_5)=d_1([e_2,e_2])=[d_1(e_2),e_2]+[e_2,d_1(e_2)]=\sum\limits_{t=5}^{n-1}\delta b_{t-3}e_t+b_1e_3+\delta b_2e_5.$
On the other hand, we have
\[d_1(\delta e_5)=(3a_1+b_2)\delta e_5+\sum\limits_{t=6}^{n-1}\delta b_{t-3}e_t.\]
Consequently, $b_1=0, \ \delta b_2=3\delta a_1.$

The description of the matrix form of derivations for the algebra $M^{2,\lambda}$ is obtained similar to the above.
\end{proof}

\begin{prop} \label{prop2}  Any derivation $d$ of the algebra $M^{3,1}$ has the following matrix form:
\[d=\left(\begin{array}{ccccccc}
a_1&0&a_{3}& a_{4}&a_{5}&a_{6}\\
0&3a_1&b_{3}&b_{4}&b_{5}&b_{6}\\
0&0&4a_{1}&b_{3}&b_{4}&0\\
0&0&0&5a_{1}&b_{3}&0\\
0&0&0&0&6a_1&0\\
0&0&0&0&0&2a_1\\
\end{array}\right).\]
\end{prop}
\begin{proof}
The proof is established by straightforward calculations of the derivation property and the table of multiplications of the algebra $M^{3,1}$.
\end{proof}

Propositions \ref{prop1} and \ref{prop2} imply the possible dimensions for complemented spaces.

\begin{cor}\label{cor111}

The following holds:

\begin{itemize}
    \item If $R(M^{1,0},s)=M^{1,0}\oplus Q$, then $s\leq2$;
    \item if $R(M^{1,1},s)=M^{1,1}\oplus Q$, then $s=1$;
    \item if $R(M^{2,\lambda},s)=M^{2,\lambda}\oplus Q$, then $s\leq2$;
    \item if $R(M^{3,1},s)=M^{3,1}\oplus Q$, then $s=1$.
\end{itemize}

\end{cor}

\subsection{Descriptions of algebras $R(M^{i,*}, 1), \ i=1, 2, 3$}

\

\

In this subsection we describe solvable Leibniz algebras $R(M^{i,*}, 1)=M^{i,*}\oplus Q, \ i=1,2,3$.

 \begin{thm} \label{thmmu1} An arbitrary algebra of the family $R(M^{1,0}, 1)$ admits a basis $\{e_1, e_2, \dots, e_{n},x\}$ such that its table of multiplications is one of the following types:
$$\begin{array}{l}
R_1(M^{1,0},1)(\alpha_2,\dots, \alpha_{n}):\\[1mm]
\begin{cases}
[e_1,x]=\alpha_2e_n,&  [x,x]=\alpha_n e_{n},\\[1mm]
[e_i,x]=e_i+\sum\limits_{t=i+1}^{n-1}\alpha_{t-i+2}e_t,& 2\leq i\leq n-1,\\[1mm]
\end{cases}\end{array}\quad
\begin{array}{l}
R_2(M^{1,0},1)(\alpha):\\[1mm]
\begin{cases}
[e_1,x]=e_1,\\[1mm]
[e_i,x]=(i-1)e_i,& 2\leq i\leq n-1,\\[1mm]
[e_n,x]=2e_{n}, & [x,e_1]=-e_1+\alpha e_2,\\[1mm]
\end{cases}\end{array}$$

$$\begin{array}{l}
R_3(M^{1,0},1)(\alpha):\\[1mm]
\begin{cases}
[e_1,x]=e_1,& [e_2,x]=2e_2+ \alpha e_n,\\[1mm]
[e_i,x]=ie_i,& 3\leq i\leq n-1,\\[1mm]
[e_n,x]=2e_n,& [x,e_1]=-e_1,\\[1mm]
\end{cases}\end{array}\quad\quad\quad\quad\quad
\begin{array}{l}
R_4(M^{1,0},1)(\alpha):\\[1mm]
\begin{cases}
[e_1,x]=e_1+\alpha e_{n-2},& [x,x]=-\alpha e_{n-3},\\[1mm]
[e_i,x]=(i+3-n)e_i,& 2\leq i\leq n-1,\\[1mm]
[e_n,x]=\alpha e_{n-1}+2e_{n},& [x,e_1]=-e_1,\\[1mm]
\end{cases}\end{array}$$

$$\begin{array}{l}
R_5(M^{1,0},1)(\alpha):\\[1mm]
\begin{cases}
[e_1,x]=e_1+\alpha e_{n-1}, \\[1mm]
[e_i,x]=(i+2-n)e_i,& 2\leq i\leq n,\\[1mm]
[x,e_1]=-e_1,&  [x,x]=-\alpha e_{n-2},\\[1mm]
\end{cases}\end{array}\quad
\begin{array}{l}
R_6(M^{1,0},1)(\alpha):\\[1mm]
\begin{cases}
[e_1,x]=e_1,&[x,e_1]=-e_1,\\[1mm]
[e_i,x]=(i+1-n)e_i,& 2\leq i\leq n-1,\\[1mm]
[e_n,x]=2e_n,& [x,x]=\alpha e_{n-1},\\[1mm]
\end{cases}\end{array}$$

$$\begin{array}{l}
R_7(M^{1,0},1)(\alpha):\\[1mm]
\begin{cases}
[e_1,x]=e_1,& [x,e_1]=-e_1, \\[1mm]
[e_i,x]=(i-2+\alpha)e_i,& 2\leq i\leq n-1,\\[1mm]
[e_n,x]=2e_n,& \alpha\notin \{1,2,3-n,4-n,5-n\}.\\[1mm]
\end{cases}\end{array}\quad\quad\quad\quad\quad\quad\quad\quad\quad\quad\quad\quad\quad\quad$$
\end{thm}
\begin{proof} Let $\{e_1, e_2, \dots, e_{n},x\}$ be a basis such that the table of multiplications of an algebra $M^{1,0}$ in the basis $\{e_1, e_2, \dots, e_{n}\}$ has the form of Theorem \ref{thmquasi-fil}.

From Proposition \ref{prop1} we have the products in the algebra $R(M^{1,0}, 1)$:
$$\begin{cases}
[e_1,x]=a_{1}e_1+a_{n-2}e_{n-2}+a_{n-1}e_{n-1}+a_{n}e_n, \\[1mm]
[e_2,x]=\sum\limits_{t=2}^{n}b_{t}e_t,\\[1mm]
[e_i,x]=((i-2)a_{1}+b_{2})e_i+\sum\limits_{t=i+1}^{n-1}b_{t-i+2}e_t,\ 3\leq i\leq n-1,\\[1mm] [e_n,x]=a_{n-2}e_{n-1}+2a_{1}e_{n},  \\[1mm]
[x,e_i]=\sum\limits_{t=1}^{n}c_{i,t}e_t,\ 1\leq i\leq n-1,\  [x,x]=\sum\limits_{t=1}^{n}\delta_{t}e_t,\\[1mm]
\end{cases}$$
where $a_1\neq 0$ or $b_2\neq 0$.

Since $[x,e_1]+[e_1,x],\ [x,x], \ e_n\in \Ann_r(R(M^{1,0}, 1)),$ we conclude $$c_{1,1}=-a_1,\ \delta_1=0,\ [x,e_n]=0.$$

Taking the change
$$x^\prime=x-\sum\limits_{t=3}^{n-1}c_{1,t}e_{t-1}-c_{1,n}e_1,$$
one can assume $c_{1,t}=0$ for $3\leq t\leq n.$

Considering the equalities ${\mathcal L}(x,e_1,e_2)={\mathcal L}(x,e_i,e_1)={\mathcal L}(x,e_2,x)={\mathcal L}(x,x,e_2)=0$ with $2 \leq i\leq n-2$ we derive restrictions:
$$c_{2,t}=0,\ \ 1\leq t\leq n-2, \quad c_{i,t}=0,\ \ 3\leq i\leq n-1,\ \ 1\leq t\leq n,$$
$$2a_1c_{2,n}=b_2c_{2,n}=b_2c_{2,n-1}=(n-3)a_1c_{2,n-1}+a_{n-2}c_{2,n}=0.$$

We have $c_{2,n}=0$ (otherwise we have $a_1=b_2=0$), then from the above restrictions we obtain $c_{2,n-1}=0.$

The equality ${\mathcal L}(x,e_{1},x)=0$ implies
\begin{equation}\label{eq2}\left\{\begin{array}{lll}
c_{1,2}(a_1-b_2)=0,&  \delta_t=c_{1,2}b_{t+1}, & 2\leq t\leq n-4,\\[1mm]
\delta_{n-3}=c_{1,2}b_{n-2}-a_1a_{n-2}, & \delta_{n-2}=c_{1,2}b_{n-1}-a_1a_{n-1},& c_{1,2}b_{n}=a_1a_{n}.\\[1mm]
\end{array}\right.\end{equation}

Thus, the table of multiplications of the algebra $ R(M^{1,0},1)$ has form:
\begin{equation}\label{eq111}
\begin{cases}
[e_1,x]=a_1e_1+a_{n-2}e_{n-2}+a_{n-1}e_{n-1}+a_{n}e_n,\ [e_2,x]=\sum\limits_{t=2}^{n}b_{t}e_t,\\[1mm]
[e_i,x]=((i-2)a_{1}+b_{2})e_i+\sum\limits_{t=i+1}^{n-1}b_{t-i+2}e_t,\ 3\leq i\leq n-1,\\[1mm]
[e_n,x]=a_{n-2}e_{n-1}+2a_1e_{n},\ [x,e_1]=-a_1e_1+c_{1,2}e_{2},\ [x,x]=\sum\limits_{t=2}^{n}\delta_{t}e_t.\\[1mm]
\end{cases}  \end{equation}

Let us take the general change of generator basis elements:
$$e_1^\prime=\sum\limits_{i=1}^{n}A_{i}e_i,\ \ \ e_2'=\sum\limits_{t=1}^{n}B_{t}e_t,\ \ \ x^\prime=Hx+\sum\limits_{t=1}^{n}C_{t}e_t.$$

We express the rest basis elements $e_i', 3\leq i \leq n$ from the products $[e_i', e_1']=e_{i+1}'$.

We consider the table of multiplications (\ref{eq111}) in terms of the basis $\{e_1', \dots, e_n', x'\}$ and with parameters $a_i', b_i', \delta_i'$. Then express left sides and right sides of the products with respect to basis $\{e_1, \dots, e_n, x\}$, we obtain the following relations between parameters
$a_i', b_i', \delta_i'$ and $a_i, b_i, \delta_i$:
$$a_1'=Ha_1,\ \ \ b_{2}'=Hb_2.$$

%

Consider the following possible cases.

{\bf Case 1.} Let $b_2\neq0$.

{\bf Case 1.1.} Let $a_1=0$. Then $b_2\neq 0$ and by choosing $H=\frac{1}{b_2}$ we can assume $b_2=1.$
From restrictions (\ref{eq2}) we have
\[c_{1,2}=\delta_{t}=0, \ 2\leq t\leq n-2.\]

Now applying the change of basis elements $e_1, e_2, e_n$ and $x$ as follows:
$$e_1^\prime=e_1-a_{n-2}e_{n-2}-(a_{n-1}-a_{n-2}b_3+a_{n}a_{n-2})e_{n-1},\ \ e_{2}'=e_2+b_{n}e_n,$$
$$e_{n}'=e_n-a_{n-2}e_{n-1},\ \ \ x'=x-\delta_{n-1}e_{n-1},$$
we can assume $a_{n-2}=a_{n-1}=b_n=\delta_{n-1}=0$.

Thus, we obtain the family of algebras $R_1(M^{1,0},1)(\alpha_2, \dots, \alpha_{n}).$

{\bf Case 1.2.} Let $a_1\neq0$. Then by rescaling $x^\prime=\frac{1}{a_1}x$ in (\ref{eq111}) we can assume $a_{1}=1$. Note that restrictions (\ref{eq2}) in this case have the following form:
\begin{equation}\label{eq3}\left\{\begin{array}{lll}
c_{1,2}(1-b_2)=0,& \delta_t=c_{1,2}b_{t+1},& 2\leq t\leq n-4,\\[1mm]
\delta_{n-3}=c_{1,2}b_{n-2}-a_{n-2}, & \delta_{n-2}=c_{1,2}b_{n-1}-a_{n-1},& c_{1,2}b_{n}=a_{n}.
\\[1mm]
\end{array}\right.\end{equation}

Now we investigate possible cases for parameter $b_2$.

\begin{itemize}
\item Let $b_2=1$. Then applying the basis transformation in the following form:
$$e_1^\prime=e_1, \ e_2^\prime=e_2+\sum\limits_{j=3}^{n}A_{j}e_{j}, \ e_i^\prime=e_i+\sum\limits_{j=i+1}^{n-1}A_{j-i+2}e_{j}, \ 3\leq i\leq n-1,$$ with
$$A_{3}=-b_{3},  \ A_{i}=-\frac{1}{i-2}\Big(b_{i}+\sum\limits_{j=3}^{i-1}b_{i-j+2}A_{j}\Big),\ 4\leq i\leq n-2,$$
$$A_{n-1}=-\frac{1}{n-3}\Big(b_{n-1}+\sum\limits_{j=3}^{n-2}b_{n+1-j}A_{j}-b_na_{n-2}\Big), \ A_{n}=-b_{n},$$
we obtain $b_{t}=0$ for $3\leq t\leq n.$

From restrictions (\ref{eq3}) we get
$$\delta_{n-3}=-a_{n-2},\ \delta_{n-2}=-a_{n-1}, \ a_{n}=\delta_t=0, \ 2\leq t\leq n-4,$$
consequently, $[x,e_1]=-e_1+\sum\limits_{t=2}^{n}c_{1,t}e_t.$

Taking the change
$$e_1^\prime=e_1-\frac{a_{n-2}}{n-4}e_{n-2}-\frac{a_{n-1}}{n-3}e_{n-1},\ \ e_n^\prime=e_n-\frac{a_{n-2}}{n-4}e_{n-1},$$
we obtain $a_{n-2}=a_{n-1}=0$.

The equality $\mathcal{L}(x,e_{1},x)=0$ implies $c_{1,t}=0, \ 3\leq t\leq n.$

Finally, making a change $x^\prime=x-\frac{\delta_{n-1}}{n-2}e_{n-1}-\frac{\delta_{n}}{2}e_n$, we derive $\delta_{n-1}=\delta_{n}=0$ and the algebra $R_2(M^{1,0},1)(\alpha)$ is obtained.

\item Let $b_2\neq1$. Then from (\ref{eq3}) we conclude $a_n=0,\ c_{1,2}=0,\ \delta_{t}=0,\ 2\leq t\leq n-4.$

The result of the change of the basis elements $\{e_1, \dots, e_{n-1}\}$:
$$e_1^\prime=e_1, \ e_i^\prime=e_i+\sum\limits_{j=i+1}^{n-1}A_{j-i+2}e_{j},\ 2\leq i\leq n-1,$$ with
$$A_{3}=-b_{3},\ A_{i}=-\frac{1}{i-2}\Big(b_{i}+\sum\limits_{j=3}^{i-1}b_{i-j+2}A_{j}\Big),\ 4\leq i\leq n-1,$$
is $b_{t}=0, \ 3\leq t\leq n-1.$

Taking $x^\prime=x-\frac{1}{2}\delta_{n}e_n$, we have
$\delta_{n}=0$ and the table of multiplications (\ref{eq111}) has the following form:
\begin{equation}\label{eq222}
\left\{\begin{array}{ll}
[e_1,x]=e_1+a_{n-2}e_{n-2}+a_{n-1}e_{n-1},&\\[1mm]
[e_2,x]=b_2e_2+b_{n}e_n,&\\[1mm]
[e_i,x]=(i-2+b_{2})e_i,&3\leq i\leq n-1,\\[1mm]
[e_n,x]=a_{n-2}e_{n-1}+2e_{n},&\\[1mm]
[x,e_1]=-e_1,&\\[1mm]
[x,x]=-a_{n-2}e_{n-3}-a_{n-1}e_{n-2}+\delta_{n-1}e_{n-1}.&\\[1mm]
\end{array}\right.
\end{equation}
\begin{itemize}
  \item  Let $b_2=2$. Then taking  the  change of elements $\{e_1, e_2, e_n, x\}$ in (\ref{eq222}) as follows:
$$e_1^\prime=e_1-\frac{a_{n-2}}{n-3}e_{n-2}-\frac{a_{n-1}}{n-2}e_{n-1},\ e_2^\prime=e_2-\frac{a_{n-2}b_n}{(3-n)^2}e_{n-1},  $$
$$e_n^\prime=e_n-\frac{a_{n-2}}{n-3}e_{n-1},\ x'=x-\frac{\delta_{n-1}}{n-1}e_{n-1}+\frac{a_{n-2}}{n-3}e_{n-3}+\frac{a_{n-1}}{n-2}e_{n-2},$$
we can assume that $a_{n-2}=a_{n-1}=\delta_{n-1}=0$. Hence, we obtain the family of algebras $R_3(M^{1,0},1)(\alpha).$

\item  Let $b_2=5-n$. Then setting
$$e_1^\prime=e_1-a_{n-1}e_{n-1},\ e_2^\prime=e_2+\frac{a_{n-2}b_n}{(3-n)^2}e_{n-1}+\frac{b_{n}}{3-n}e_{n},\ x^\prime=x-\frac{\delta_{n-1}}{2}e_{n-1}+a_{n-1}e_{n-2},$$
in (\ref{eq222}) one can get $a_{n-1}=b_n=\delta_{n-1}=0$. So, we obtain the family of algebras $R_4(M^{1,0},1)(\alpha).$

 \item Let $b_2=4-n$. Then putting
$$e_1^\prime=e_1+a_{n-2}e_{n-2},\ e_2^\prime=e_2+\frac{a_{n-2}b_n}{(2-n)(3-n)}e_{n-1}+\frac{b_n}{2-n}e_n,$$
$$e_n^\prime=e_n+a_{n-2}e_{n-1}, \ x'=x-\delta_{n-1}e_{n-1}-a_{n-2}e_{n-3},$$
in (\ref{eq222}) we derive $a_{n-2}=b_n=\delta_{n-1}=0$. Therefore, the family $R_5(M^{1,0},1)(\alpha)$ is obtained.

 \item Let $b_2=3-n$. Then applying the change in (\ref{eq222})
$$e_1^\prime=e_1+\frac{a_{n-2}}{2}e_{n-2}+a_{n-1}e_{n-1},\ e_2^\prime=e_2+\frac{a_{n-2}b_n}{(1-n)(3-n)}e_{n-1}+\frac{b_n}{1-n}e_n,$$
$$e_n^\prime=e_n+\frac{a_{n-2}}{2}e_{n-1},\ x^\prime=x-\frac{a_{n-2}}{2}e_{n-3}-a_{n-1}e_{n-2},$$
we can assume $a_{n-2}=a_{n-1}=b_n=0.$ Hence, we get $R_6(M^{1,0},1)(\alpha).$

 \item Let $b_2\neq 2,\ 3-n,\ 4-n,\ 5-n$. Then setting
$$e_1^\prime=e_1-\frac{a_{n-2}}{n-5+b_2}e_{n-2}-\frac{a_{n-1}}{n-4+b_2}e_{n-1},\ e_2^\prime=e_2+\frac{a_{n-2}b_n}{(b_2-2)(3-n)}e_{n-1}+\frac{b_n}{b_2-2}e_n,$$
$$e_n^\prime=e_n-\frac{a_{n-2}}{n-5+b_2}e_{n-1},\ x'=x-\frac{\delta_{n-1}}{n-3+b_2}e_{n-1}+\frac{a_{n-2}}{n-5+b_2}e_{n-3}+\frac{a_{n-1}}{n-4+b_2}e_{n-2},$$
in (\ref{eq222}) we get  $a_{n-2}=a_{n-1}=b_n=\delta_{n-1}=0.$ Thus, we obtain the algebra $R_7(M^{1,0},1)(\alpha).$
\end{itemize}
\end{itemize}

{\bf Case 2.} Let $b_2=0$. Then $a_1\neq 0$ and putting $H=\frac{1}{a_1}$ one can assume $a_1=1$.

Taking the change of basis elements $\{e_1, \dots, e_{n-1}\}$ as follows
$$e_1^\prime=e_1,\ e_2^\prime=e_2+\sum\limits_{j=3}^{n}A_{j}e_{j},\ e_i^\prime=e_i+\sum\limits_{j=i+1}^{n-1}A_{j-i+2}e_{j},\ 3\leq i\leq n-1,
$$
with
$$A_{3}=-b_{3}, \ \ \ A_{i}=-\frac{1}{i-2}\Big(b_{i}+\sum\limits_{j=3}^{i-1}b_{i-j+2}A_{j}\Big), \ \ 4\leq i\leq n-2,$$
$$A_{n-1}=-\frac{1}{n-3}\Big(b_{n-1}+\sum\limits_{j=3}^{n-2}b_{n+1-j}A_{j}+A_na_{n-2}\Big),\ \ \ A_{n}=-\frac{1}{2}b_{n},$$
we can assume $b_{t}=0, \ 3\leq t\leq n.$

Putting
\[e_1^\prime=e_1-\frac{a_{n-2}}{n-5}e_{n-2}-\frac{a_{n-1}-a_na_{n-2}}{n-4}e_{n-1}-a_{n}e_n,\ \ e_n^\prime=e_n-\frac{a_{n-2}}{n-5}e_{n-1},\]
we get $a_{n-2}=a_{n-1}=a_{n}=0$.

Now, setting $x^\prime=x-\sum\limits_{t=3}^{n-1}c_{1,t}e_{t-1}$, we get $c_{1,t}=0, \ 3\leq t\leq n-1$.

The equality $\mathcal{L}(x,e_{1},x)=0$ implies
$$c_{1,2}=0,\ \ c_{1,n}=0, \ \delta_t=0,\ \ 2\leq t\leq n-2.$$

Putting $x^\prime=x-\frac{\delta_{n-1}}{n-3}e_{n-1}-\frac{\delta_{n}}{2}e_n$ if necessary, we obtain
$\delta_{n-1}=0,\ \delta_{n}=0 $ and we get the algebra of the family $R_7(M^{1,0},1)(\alpha).$
\end{proof}

In the following theorem we investigate the isomorphism inside the families of algebras of Theorem \ref{thmmu1}.

\begin{thm} \label{thm35} An arbitrary algebra of the family $R(M^{1,0},1)$ is isomorphic to one of the following pairwise non-isomorphic algebras:
$$R_1(M^{1,0},1)(\alpha_2,\dots, \alpha_{n}),\ R_i(M^{1,0},1)(1), \ 2\leq i \leq 6, \ R_7(M^{1,0},1)(\alpha),$$
where $\alpha_i, \alpha\in \mathbb{C}$ and the first non-vanishing parameter $\{\alpha_2,\dots, \alpha_{n}\}$
in the algebra  $R_1(M^{1,0},1)(\alpha_2,\dots, \alpha_{n})$ can be scaled to $1$.
\end{thm}
\begin{proof}
Let us consider the general change of generator basis elements of the algebra $R(M^{1,0},1)$:
\[e_1^\prime=\sum\limits_{t=1}^{n}A_{t}e_t, \quad e_2^\prime=\sum\limits_{t=1}^{n}B_{t}e_t,\quad x^\prime=Hx+\sum\limits_{t=1}^{n}C_{t}e_t.\]

From the products
\[[e_i^\prime,e_1^\prime]=e_{i+1}^\prime, \quad 2\leq i\leq n-1,\] we derive
\[e_{3}^\prime=A_1\sum\limits_{i=3}^{n-1}B_{i-1}e_i+A_1B_{1}e_{n},\quad
e_{i}^\prime=A_1^{i-2}\sum\limits_{t=i}^{n-1}B_{t-i+2}e_t,\ \ 4\leq i\leq n-1,\quad
e_{n}^\prime=A_1\sum\limits_{t=3}^{n-1}A_{t-1}e_t+A_1^2e_n.\]

Considering
\[[e_1^\prime,e_2^\prime]=[e_n^\prime,e_1^\prime]=0,\] we deduce
\[B_1=A_t=0,\quad 2\leq t\leq n-3.\]

Thus we have the following:
$$e_1^\prime=A_{1}e_1+A_{n-2}e_{n-2}+A_{n-1}e_{n-1}+A_ne_n,\ \ \ e_2'=B_2e_2+\sum\limits_{t=3}^{n}B_{t}e_t,\ \ \ x^\prime=Hx+\sum\limits_{t=1}^{n}C_{t}e_t.$$
$$e_{i}^\prime=A_1^{i-2}\sum\limits_{t=i}^{n-1}B_{t-i+2}e_t,\ \ 3\leq i\leq n-1, \ \ e_{n}^\prime=A_1A_{n-2}e_{n-1}+A_1^2e_n $$

\begin{itemize}
\item Let us consider the family $R_1(M^{1,0},1)(\alpha_2,\dots, \alpha_{n}).$

Considering the below presented products we obtain restrictions:
$$\left\{\begin{array}{lll}
[x^\prime,e_1^\prime]=0,& \Rightarrow & C_{i}=0, \ 1\leq i\leq n-2,\\[1mm]
[e_{n-1}^\prime,x^\prime]=e_{n-1}^\prime, & \Rightarrow & H=1, \\[1mm]
[e_1^\prime,x^\prime]=\alpha_{2}^\prime e_n^\prime, & \Rightarrow & A_{n-2}=A_{n-1}=0,\ \ \ \alpha_{2}'=\frac{\alpha_{2}}{A_1}, \\[1mm]
[x^\prime,x^\prime]=\alpha_{n}^\prime e_n^\prime, & \Rightarrow & C_{n-1}=0, \ \alpha_{n}'=\frac{\alpha_{n}}{A_1^2}, \\[1mm]
[e_2^\prime,x^\prime]=e_2'+\sum\limits_{t=3}^{n-1}\alpha_{t}'e_t', & \Rightarrow & \alpha_{t}'=\frac{\alpha_t}{A_1^{t-2}},\ \ \ 3\leq t\leq n-1.  \\[1mm]
\end{array}\right.$$

Thus, we deduce the following invariant relations:
$$\alpha_{2}'=\frac{\alpha_{2}}{A_1}, \ \alpha_{n}'=\frac{\alpha_{n}}{A_1^2},\ \alpha_{t}'=\frac{\alpha_t}{A_1^{t-2}},\ \ 3\leq t\leq n-1.$$

From these relations we conclude that by suitable value of $A_1$ the first non-vanishing parameter $\{\alpha_2,\dots, \alpha_{n}\}$ of the algebra from $R_1(M^{1,0},1)(\alpha_2,\dots, \alpha_{n})$ can be scaled to $1$.

\item  Consider the family of algebras $R_2(M^{1,0},1)(\alpha)$.

Similarly, we have
$$\left\{\begin{array}{lll}
[e_{n}^\prime,x^\prime]=2e_{n}^\prime,& \Rightarrow & H=1,\ \ A_{n-2}=0,\\[1mm]
[x^\prime,e_1^\prime]=-e_1'+\alpha^\prime e_{2}^\prime, & \Rightarrow & A_{n}=-A_1C_1, \ A_{n-1}=-A_1C_{n-2},\\[1mm]
&& B_t=0, \ 3\leq t\leq n, \ C_t=0, \ 2\leq t\leq n-2,\ \alpha'=\frac{\alpha A_1}{B_2}.\\[1mm]
\end{array}\right.$$

If $\alpha\neq0$, then by putting $B_2=A_1 \alpha$, we get $\alpha'=1$ and we have the algebra $R_2(M^{1,0},1)(1)$;

If $\alpha=0$, we obtain the algebra $R_7(M^{1,0},1)(1)$.

\item Consider the family $R_3(M^{1,0},1)(\alpha)$. Then considering the products

\[[e_{n}^\prime,x^\prime]=2e_{n}^\prime,\quad [e_2^\prime,x^\prime]=2e_2'+\alpha'e_{n}',\]
we get
\[H=1,\quad A_{n-2}=(t-2)B_t+C_1B_{t-1}=0,\quad 3\leq t\leq n-1,\quad \alpha'=\frac{\alpha B_2}{A_{1}^{2}}.\]

If $\alpha\neq0$, then putting $B_2=\frac{A_1^2}{\alpha}$, we have $\alpha'=1$ and hence we obtain the algebra $R_3(M^{1,0},1)(1)$;

If $\alpha=0$, then we derive $R_7(M^{1,0},1)(2)$.

\item Consider the algebra $R_4(M^{1,0},1)(\alpha)$. From the products
\[[e_{n-1}^\prime,x^\prime]=2e_{n-1}^\prime, \quad [e_n^\prime,x^\prime]=\alpha'e_{n-1}'+2e_{n}^\prime,\] we derive
\[H=1, \quad \alpha'=\frac{\alpha}{A_{1}^{n-5}B_2}.\]

If $\alpha\neq0$, then setting $B_2=\frac{\alpha}{A_1^{n-5}}$, we get $\alpha'=1$ and the algebra $R_4(M^{1,0},1)(1)$;

If $\alpha=0$, then we obtain algebra $R_7(M^{1,0},1)(5-n)$.

\item Now consider the case of family of algebras $R_5(M^{1,0},1)(\alpha)$.  From the products
\[[e_{n}^\prime,x^\prime]=2e_{n}^\prime,\quad [e_1^\prime,x^\prime]=e_1'+\alpha'e_{n-1}',\]
we obtain
\[H=1,\quad A_{n-2}=0, \quad A_{n}=-A_1C_1,\quad \alpha'=\frac{\alpha}{A_{1}^{n-4}B_2}.\]

If $\alpha\neq0$, then taking $B_2=\frac{\alpha}{A_1^{n-4}},$ we have $\alpha'=1$ and the algebra $R_5(M^{1,0},1)(1)$;

If $\alpha=0$, then we derive $R_7(M^{1,0},1)(4-n)$.

\item The case of the family $R_6(M^{1,0},1)(\alpha)$.  From
\[[e_{n}^\prime,x^\prime]=2e_{n}^\prime, \ [e_{1}^\prime,x^\prime]=e_{1}^\prime,\ [x^\prime,e_1^\prime]=-e_{1}^\prime,\ [x^\prime,x^\prime]=\alpha'e_{n-1}'\]
we obtain
\[H=1, \ A_{n}=-A_1C_1, \  A_{n-2}=A_{n-1}=C_{t}=0,\ 2\leq t\leq n-2,\]\[2C_{n}=-C_1^2,\ \alpha'=\frac{\alpha}{A_{1}^{n-3}B_2}.\]

If $\alpha\neq0$, then taking $B_2=\frac{\alpha}{A_{1}^{n-3}},$ we can assume $\alpha'=1$ and we obtain $R_6(M^{1,0},1)(1)$;

If $\alpha=0$, we derive the algebra $R_7(M^{1,0},1)(3-n)$.

\item Finally, consider the algebra $R_7(M^{1,0},1)(\alpha).$ Then from the products
\[[e_{n}^\prime,x^\prime]=2e_{n}^\prime,\quad [e_{n-1}^\prime,x^\prime]=(n-3+\alpha')e_{n-1}',\]
we obtain $H=1, \ A_{n-2}=0, \ \alpha'=\alpha.$
\end{itemize}\end{proof}

\begin{thm} There  are no  solvable  Leibniz  algebra  with  nilradical $M^{1,1}.$

\end{thm}

\begin{proof} From Proposition \ref{prop1} we have the products in the algebra of the family $R(M^{1,1}, 1)$:
$$\begin{cases}
[e_1,x]=a_{1}e_1+a_{n-2}e_{n-2}+a_{n-1}e_{n-1}+a_{n}e_n,\ [e_2,x]=3a_1e_2+\sum\limits_{t=3}^{n}b_{t}e_t,\\[1mm]
[e_i,x]=(i+1)a_{1}e_i+\sum\limits_{t=i+1}^{n-1}b_{t-i+2}e_t,\ 3\leq i\leq n-1,\ [e_n,x]=a_{n-2}e_{n-1}+2a_{1}e_{n},\\[1mm]
[x,e_i]=\sum\limits_{t=1}^{n}c_{i,t}e_t,\ 1\leq i\leq n-1,\ [x,x]=\sum\limits_{t=1}^{n}\delta_te_t,\\[1mm]
\end{cases}$$
where $a_1\neq 0$.

Applying the Leibniz identity, we have restrictions of structure constants
$$\left\{\begin{array}{lll}
{\mathcal L}(e_1,x,e_1)=0, & \Rightarrow & c_{1,1}=-a_{1},\\[1mm]
{\mathcal L}(e_2,x,e_2)=0,& \Rightarrow & c_{2,1}=0,\ c_{2,2}=-3a_1, \\[1mm]
{\mathcal L}(x,e_1,e_2)=0,& \Rightarrow & a_1=0. \\[1mm]
\end{array}\right.$$

Thus, we get a contradiction with condition $a_1\neq0 $, which imply the non existence of a solvable Leibniz algebra with nilradical $M^{1,1}$ and one-dimensional complemented space.
\end{proof}

Similar to Theorems \ref{thmmu1} and \ref{thm35} we give the description up to isomorphism of solvable Leibniz algebras with nilradical $M^{2,\lambda}$ and one-dimensional complementary space to the nilradical, i.e., solvable Leibniz algebras $R(M^{2,\lambda},1).$

\begin{thm} An arbitrary algebra of the family $R(M^{2,\lambda},1)$ is isomorphic to one of the following pairwise non-isomorphic algebras:
$$\begin{array}{l}
R_1(M^{2,\lambda},1):\\[1mm]
\begin{cases}
[e_1,x]=e_1+e_{n-1}, & [e_2,x]=2e_2+(1+\lambda)e_n,\\[1mm]
[e_i,x]=ie_i,& 3\leq i\leq n-2,\\[1mm]
[e_{n-1},x]=e_{n-1},& [e_n,x]=2e_{n},\\[1mm]
[x,e_1]=-e_1-e_{n-1},& [x,e_{n-1}]=-e_{n-1}, \\[1mm]
[x,e_{n}]=(\lambda-1)e_n, & \lambda\in\{-1,1\},\\[1mm]
\end{cases}\end{array}
 \quad\quad
 \begin{array}{l}
 R_2(M^{2,-1},1):\\[1mm]
\begin{cases}
[e_1,x]=e_1+e_n,& [e_n,x]=e_{n},\\[1mm]
[e_i,x]=ie_i,& 2\leq i\leq n-2,\\[1mm]
[x,e_1]=-e_1,& [x,e_{n}]=-e_n,\\[1mm]
\end{cases}\end{array}
$$

$$\begin{array}{l}
R_3(M^{2,-1},1):\\[1mm]
\begin{cases}
[e_i,x]=ie_i,& 1\leq i\leq n-2,\\[1mm]
[e_{n-1},x]=-e_{n-1},& [x,e_1]=-e_1,\\[1mm]
[x,e_{n-1}]=e_{n-1},& [x,x]=e_n,\\[1mm]
\end{cases}\end{array}\quad\quad\quad\quad
\begin{array}{l}
R_4(M^{2,-1},1)(\alpha):\\[1mm]
\begin{cases}
[e_i,x]=ie_i,& 1\leq i\leq n-2,\\[1mm]
[e_{n-1},x]=\alpha e_{n-1},& [e_n,x]=(1+\alpha)e_{n},\\[1mm]
[x,e_1]=-e_1,& [x,e_{n-1}]=-\alpha e_{n-1},\\[1mm]
[x,e_{n}]=-(\alpha+1)e_n,& \\[1mm]
\end{cases}\end{array}$$

$$\begin{array}{l}
R_5(M^{2,\lambda},1):\\[1mm]
\begin{cases}
[e_i,x]=ie_i, & 1\leq i\leq n-2,\\[1mm]
[e_{n-1},x]=\lambda e_{n-1}, & [e_n,x]=(1+\lambda)e_{n},\\[1mm]
[x,e_1]=-e_1, & [x,e_{n-1}]=-\lambda e_{n-1}, \\[1mm]
\lambda\notin\{-1,0,1\},&\\[1mm]
\end{cases}\end{array}\quad
\begin{array}{l}
R_6(M^{2,-1},1)(\alpha_{2},\dots,\alpha_{n}):\\[1mm]
\begin{cases}
[e_i,x]=\sum\limits_{t=i+1}^{n-2}\alpha_{t-i+1}e_t,& 1\leq i\leq n-3,\\[1mm]
[e_{n-1},x]=e_{n-1}+\alpha_{n-1}e_n,& [e_n,x]=e_{n},\\[1mm]
[x,e_{n-1}]=-e_{n-1}-\alpha_{n-1}e_n,& [x,e_{n}]=-e_n,\\[1mm]
[x,x]=\alpha_{n}e_{n-2}.&\\[1mm]
\end{cases}\end{array}$$

$$\begin{array}{l}
R_7(M^{2,0},1):\\[1mm]
\begin{cases}
[e_i,x]=ie_i,& 1\leq i\leq n-2,\\[1mm]
[e_{n-1},x]=-e_{n-1},& [x,e_1]=-e_1,\\[1mm]
[x,x]=e_n,&\\[1mm]
\end{cases}\end{array}\quad\quad
\begin{array}{l}
R_8(M^{2,0},1):\\[1mm]
\begin{cases}
[e_i,x]=ie_i, & 1\leq i\leq n-2,\\[1mm]
[e_{n-1},x]=e_{n-3}+(n-3)e_{n-1},& [x,e_1]=-e_1,\\[1mm]
[e_n,x]=e_{n-2}+(n-2)e_{n},\\[1mm]
\end{cases}\end{array}$$

$$\begin{array}{l}
R_9(M^{2,0},1):\\[1mm]
\begin{cases}
[e_i,x]=ie_i,& 1\leq i\leq n-2,\\[1mm]
[e_{n-1},x]=e_{n-2}+(n-2)e_{n-1},&[x,e_1]=-e_1,\\[1mm]
[e_n,x]=(n-1)e_{n},\\[1mm]
\end{cases}\end{array}\quad
\begin{array}{l}
R_{10}(M^{2,0},1):\\[1mm]
\begin{cases}
[e_1,x]=e_1+e_{n},&[x,e_1]=-e_1,\\[1mm]
[e_i,x]=ie_i,&2\leq i\leq n-2,\\[1mm]
[e_n,x]=e_{n},&[x,x]=-e_{n-1},\\[1mm]
\end{cases}\end{array}$$

$$\begin{array}{l}
R_{11}(M^{2,0},1)(\alpha):\\[1mm]
\begin{cases}
[e_i,x]=ie_i,& 1\leq i\leq n-2,\\[1mm]
[e_{n-1},x]=\alpha e_{n-1},& [e_n,x]=(1+\alpha)e_{n},\\[1mm]
[x,e_1]=-e_1,\\[1mm]
\end{cases}\end{array}\quad
\begin{array}{l}
R_{12}(M^{2,0},1)(\alpha_{2},\dots,\alpha_n):\\[1mm]
\begin{cases}
[e_i,x]=\sum\limits_{t=i+1}^{n-2}\alpha_{t-i+1}e_t,& 1\leq i\leq n-3,\\[1mm]
[e_{n-1},x]=e_{n-1}+\alpha_{n-1}e_n,& [e_n,x]=e_{n},\\[1mm]
[x,x]=\alpha_{n}e_{n-2}.\\[1mm]
\end{cases}\end{array}$$
where $\alpha_i, \alpha\in \mathbb{C}$ and the  first non-vanishing parameter $\{\alpha_{2},\dots,\alpha_{n}\}$
in the algebra from $R_6(M^{2,-1},1)(\alpha_{2},\dots,\alpha_{n})$ and $R_6(M^{2,0},1)(\alpha_{2},\dots,\alpha_{n})$ can be scaled to $1$.
\end{thm}

In order to complete the classification of solvable Leibniz algebras whose nilradical is quasi-filiform Leibniz algebra of maximum length and the dimension of complementary space is one-dimensional, we need to classify seven-dimensional solvable Leibniz algebras with nilradical $M^{3,1}$.

\begin{thm} An arbitrary algebra of the family $R(M^{3,1},1)$ is isomorphic to the algebra:
$$\left\{\begin{array}{ll}
[e_1,x]=-[x,e_1]=e_1,&\\[1mm]
[e_i,x]=-[x,e_i]=(i+1)e_i,&2\leq i\leq4,\\[1mm]
[e_5,x]=6e_{5},\ \ \ [e_6,x]=2e_{6}.&\\[1mm]
\end{array}\right.$$
\end{thm}
\begin{proof} Due to Corollary \ref{cor111} the dimension of complementary space to the nilradical $M^{3,1}$ is equal to one. From Proposition \ref{prop2} we have the following products in the algebra $R(M^{3,1}, 1)$:
$$\begin{cases}
[e_1,x]=e_1+a_{3}e_3+a_{4}e_4+a_{5}e_5+a_{6}e_6,& [e_2,x]=3e_2+b_{3}e_3+b_{4}e_4+b_{5}e_5+b_{6}e_6,\\[1mm]
[e_3,x]=4e_3+b_{3}e_4+b_4e_5,& [e_4,x]=5e_4+b_{3}e_5,\\[1mm]
[e_5,x]=6e_5,& [e_6,x]=2e_{6}.\\[1mm]
\end{cases}$$

Putting
$$x'=x-a_{6}e_1+a_{3}e_2+a_{4}e_3+a_{5}e_4,$$
we may assume $a_{t}=0, \ 3\leq t\leq 6$.

By taking the basis transformation as follows:
$$e_1^\prime=e_1,\
e_2^\prime=e_2-b_{3}e_{3}-\frac{1}{2}\Big(b_{4}-b_{3}^2\Big)e_{4}-\frac{1}{6}\Big(2b_{5}-3b_{3}b_{4}+b_{3}^3\Big)e_{5}+b_{6}e_{6},$$
$$e_3^\prime=e_3-b_{3}e_{4}-\frac{1}{2}\Big(b_{4}-b_{3}^2\Big)e_{5},\ e_4^\prime=e_4-b_{3}e_{5},\ e_5^\prime=e_5,\ e_6^\prime=e_6, \ x'=x$$
one can assume $b_{t}=0$ for  $3\leq t\leq 6.$

Since $e_1, e_2, e_3, e_4\notin \Ann_r(R(M^{3,1},1)),$ we may write
$$[x,e_1]=-e_1+c_{1}e_5+c_{2}e_6, \ [x,e_2]=-3e_2+c_{3}e_5+c_{4}e_6, \ [x,x]=c_{5}e_5+c_{6}e_6,$$
where $c_i$ are parameters.

From equalities $\mathcal{L}(x,e_i,e_1)=\mathcal{L}(x,e_j,e_j)=\mathcal{L}(x,e_j,x)=0$ with $i=2,3$ and $j=1,2$
we obtain
$$[x,e_3]=-4e_3, \ [x,e_4]=-5e_4, \ [x,e_5]=[x,e_6]=0, \ [x,e_1]=-e_1, \ [x,e_2]=-3e_2.$$

Further, making the change of basis element $x$ in the following way
$$x^\prime=x-\frac{1}{5}c_{5}e_{5}-\frac{1}{2}c_{6}e_{6},$$
we may assume $[x,x]=0$, which completes the proof of theorem. \end{proof}

\subsection{Descriptions of algebras $R(M^{i,*}, 2), \ i=1, 2$}

\

\

In this subsection we classify solvable Leibniz algebras whose nilradical is quasi-filiform Leibniz maximal length and the dimension of the complementary space to the nilradical is equal to two.

\begin{thm} An arbitrary algebra of the family $R(M^{1,0},2)$ is isomorphic to the following algebra:
$$R(M^{1,0},2):\left\{\begin{array}{llll}
[e_1,x_1]=e_1,& [e_i,x_1]=(i-2)e_i,& 3\leq i\leq n-1,& [e_n,x_1]=2e_{n},\\[1mm]
[x_1,e_1]=-e_1,& [e_i,x_2]=e_i,& 2\leq i\leq n-1.&\\[1mm]
\end{array}\right.$$
\end{thm}

\begin{proof} Consider the matrix form of a derivation for the case of algebra $M^{1,0}$ given in Proposition \ref{prop1}. Since parameters $a_1$ and $b_2$ are in the diagonal, we have only two nil-independent derivations which correspond to the values of $(a_1, b_2)$ as $(1,0)$ and $(0,1)$. We denote these derivations by $\mathcal{R}_{x_{1}}$ and $\mathcal{R}_{x_{2}}$, respectively.

Let $\{e_1,\dots,e_{n},x_1,x_2 \}$ be a basis of $R(M^{1,0},2)$ such that $\mathcal{R}_{x_1}$ and $\mathcal{R}_{x_2}$ are operators of right multiplications on elements $x_1$ and $x_2,$ respectively. It is known that $\mathrm{span}\{e_1,\dots,e_n,x_1\}$ forms a subalgebra of the algebra $R(M^{1,0},2)$.
Therefore, this subalgebra is isomorphic to the algebra $R_7(M^{1,0},1)(0)$ in the list of Theorem \ref{thm35}. Then the table of multiplications of algebra $R(M^{1,0},2)$ can be written in the following form:
$$\begin{cases}
[e_1,x_1]=-[x_1,e_1]=e_1,&\\[1mm]
[e_i,x_1]=(i-2)e_i,&3\leq i\leq n-1,\\[1mm]
[e_n,x_1]=2e_{n},&\\[1mm]
[e_1,x_2]=a_{n-2}e_{n-2}+a_{n-1}e_{n-1}+a_{n}e_n,\\[1mm]
[e_2,x_2]=e_2+\sum\limits_{t=3}^{n}b_{t}e_t,\\[1mm]
[e_i,x_2]=e_i+\sum\limits_{t=i+1}^{n-1}b_{t-i+2}e_t, & 3\leq i\leq n-1,\\[1mm]
[e_n,x_2]=a_{n-2}e_{n-1},\\[1mm]
[x_2,e_i]=\sum\limits_{t=1}^{n}c_{i,t}e_t, & 1\leq i\leq n-1,\\[1mm]
[x_i,x_j]=\sum\limits_{t=1}^{n}\delta_{i,j}^te_t, & 1\leq i,j\leq 2,\\[1mm]
\end{cases}$$
where $a_i, b_i, c_{i,j}, \delta_{i,j}^k$ are parameters.

Setting
$x_2^\prime=x_2-b_{3}e_1-\sum\limits_{t=3}^{n-1}c_{1,t}e_{t-1}$, we may assume $b_{3}=c_{1,t}=0, \quad 3\leq t\leq n-1.$

Considering the Leibniz identity for triples mentioned below we obtain
$$\left\{\begin{array}{lll}
{\mathcal L}(e_1,x_2,x_2)=0, & \Rightarrow & \delta_{2,2}^1=0,\\[1mm]
{\mathcal L}(e_n,x_2,x_1)=0, & \Rightarrow & a_{n-2}=0, \\[1mm]
{\mathcal L}(e_2,x_2,x_1)=0, & \Rightarrow & \delta_{2,1}^1=b_{t}=0,\ 4\leq t\leq n, \\[1mm]
{\mathcal L}(x_2,e_1,e_2)=0, & \Rightarrow & c_{2,t}=0,\ 1\leq t\leq n-2,  \\[1mm]
{\mathcal L}(x_2,e_2,x_1)=0, & \Rightarrow & c_{2,t}=0,\ n-1\leq t\leq n,  \\[1mm]
{\mathcal L}(x_2,e_{i},e_1)=0,\ 2\leq i\leq n-2, & \Rightarrow & c_{i,t}=0, \ 3\leq i\leq n-1,\ 1\leq t\leq n, \\[1mm]
{\mathcal L}(x_2,e_{1},e_1)=0, & \Rightarrow & c_{n,t}=0, \ 1\leq t\leq n, \\[1mm]
{\mathcal L}(x_2,e_{1},x_1)=0, & \Rightarrow & c_{1,2}=c_{1,n}=\delta_{2,1}^t=0,\ 2\leq t\leq n-2, \\[1mm]
{\mathcal L}(e_1,x_{2},e_1)=0, & \Rightarrow & c_{1,1}=0, \\[1mm]
{\mathcal L}(e_1,x_{2},x_1)=0, & \Rightarrow & a_{n-1}=a_{n}=0, \\[1mm]
{\mathcal L}(x_1,e_{1},x_2)=0, & \Rightarrow & \delta_{1,2}^t=0,\ 2\leq t\leq n-2, \\[1mm]
{\mathcal L}(x_2,e_{1},x_2)=0, & \Rightarrow & \delta_{2,2}^t=0, \ 2\leq t\leq n-2. \\[1mm]

\end{array}\right.$$

Putting $x_2^\prime=x_2-\frac{\delta_{2,1}^{n-1}}{n-3}e_{n-1}-\frac{\delta_{2,1}^n}{2}e_n,$
we get $\delta_{2,1}^{n-1}=\delta_{2,1}^{n}=0.$

The equalities ${\mathcal L}(x_i,x_{2},x_1)=0$ for $1\leq i\leq2$ imply $\delta_{i,2}^{n-1}=\delta_{i,2}^{n}=0, \ 1\leq i\leq2$.

The proof of the theorem is complete.
\end{proof}

In a similar way we obtain the description of solvable algebras $R(M^{2,\lambda},2)$.
\begin{thm} An arbitrary algebra of the family $R(M^{2,\lambda},2)$ is isomorphic to one of the following non-isomorphic algebras:
$$
R(M^{2,0},2):\left\{\begin{array}{llll}
[e_i,x_1]=ie_i, & 1\leq i\leq n-2,& [e_{n},x_1]=e_n,& [x_1,e_1]=-e_1,\\[1mm]
[e_{n-1},x_2]=e_{n-1},& [e_{n},x_2]=e_n,&&\\[1mm]
\end{array}\right.$$$$
R(M^{2,-1},2):\left\{\begin{array}{llll}
[e_i,x_1]=ie_i, & 1\leq i\leq n-2,& &\\[1mm]
[x_1,e_1]=-e_1,& [e_{n},x_1]=e_n,&[x_1,e_n]=-e_n,&\\[1mm]
[e_{n-1},x_2]=e_{n-1},&[x_2,e_{n-1}]=-e_{n-1},& [e_{n},x_2]=e_n,&[x_2,e_n]=-e_n.\\[1mm]
\end{array}\right.$$
\end{thm}

\section*{Acknowledgements}

This work was supported by Agencia Estatal de Investigaci\'{o}n (Spain), grant MTM2016-79661-P (European FEDER support included, UE).

\end{document}